\documentclass{amsart}
\usepackage{amscd,graphicx,epsfig}
\usepackage{amssymb}
\usepackage[small,nohug,heads=vee]{diagrams}
\usepackage{tikz-cd}
\usepackage[colorlinks=true, pdfstartview=FitV, linkcolor=blue, citecolor=blue, urlcolor=blue]{hyperref}

\title[Regularization of Rational Group Actions]{Regularization of Rational Group Actions}
\author{Hanspeter Kraft}
\address{Departement Mathematik und Informatik, Universit\"at Basel
\newline\indent Spiegelgasse~1, CH-4051 Basel}
\email{Hanspeter.Kraft@unibas.ch}
\date{October 2016, with additions from June 2018} 

\newtheorem{thm}{Theorem}
\newtheorem*{thm*}{Theorem}
\newtheorem*{conj*}{Conjecture}

\newtheorem{prop}{Proposition}
\newtheorem{lem}{Lemma}
\newtheorem*{lem*}{Lemma}
\newtheorem{cor}{Corollary}
\newtheorem*{cor*}{Corollary}

\theoremstyle{definition}
\newtheorem{defn}{Definition}

\theoremstyle{remark}
\newtheorem*{rem*}{Remark}
\newtheorem{rem}{Remark}

\DeclareMathOperator{\Char}{char}
\DeclareMathOperator{\id}{id}
\DeclareMathOperator{\PSL}{PSL}
\DeclareMathOperator{\Aut}{Aut}
\DeclareMathOperator{\BR}{\mathcal{BR}}
\DeclareMathOperator{\pr}{pr}
\DeclareMathOperator{\Bir}{Bir}
\DeclareMathOperator{\Dom}{Dom}
\DeclareMathOperator{\Breg}{Breg}

\newcommand{\into}{\hookrightarrow}
\newcommand{\simto}{\xrightarrow{\sim}}
\newcommand{\PP}{{\mathbb P}}
\newcommand{\PSLtwo}{\PSL_{2}}

\newcommand{\Aone}{\AA^{1}}
\newcommand{\Atwo}{\AA^{2}}
\newcommand{\OOO}{\mathcal O}

\newcommand{\be}{\begin{enumerate}}
\newcommand{\ee}{\end{enumerate}}

\newcommand{\name}[1]{\textsc{#1\/}}

\newcommand{\Ptwo}{{\PP^{2}}}
\newcommand{\Pone}{{\PP^{1}}}

\newcommand{\dto}{\dashrightarrow}
\newcommand{\bk}{\Bbbk}
\newcommand{\kst}{\bk^{*}}
\newcommand{\ps}{\par\smallskip}

\renewcommand{\subset}{\subseteq}
\newcommand{\Ga}{G_{a}}
\newcommand{\Xr}{X_{\text{\rm reg}}}

\newcommand{\bX}{\tilde X}
\newcommand{\Xnull}{X^{(0)}}
\newcommand{\Xone}{X^{(1)}}
\newcommand{\bXi}{{\tilde X}^{(i)}}
\newcommand{\Xj}{X^{(j)}}
\newcommand{\bXj}{{\tilde X}^{(j)}}
\newcommand{\bXnull}{{\tilde X}^{(0)}}
\newcommand{\Xm}{X^{(m)}}
\newcommand{\rhoi}{\rho^{(i)}}
\newcommand{\brho}{\bar\rho}
\newcommand{\trho}{\tilde\rho}

\begin{document}

\begin{abstract}
We give a modern proof of the Regularization Theorem of \name{Andr\'e Weil} which says that for every
rational action of an algebraic group $G$ on a variety $X$ there exist a variety $Y$ with a regular action of $G$ and 
a $G$-equivariant birational map $X \dto Y$. Moreover, we show that a rational action of $G$ on an affine variety $X$ with the property that each $g$ from a dense subgroup of $G$ induces a regular automorphism of $X$, is a regular action.
\end{abstract}

\maketitle

\addtocounter{section}{1}

The aim of this note is to give a modern proof of the following {\it Regularization Theorem\/} due to \name{Andr\'e Weil}, see \cite{We1955On-algebraic-group}. We will follow the approach in \cite{Za1995Regularization-of-}. Our base field $\bk$ is algebraically closed. A {\it variety\/} is an algebraic $\bk$-variety,  and an {\it algebraic group\/} is an algebraic $\bk$-group.

\begin{thm}\label{thm1}
Let $G$ be an algebraic group and $X$ a variety with a rational action of $G$. Then there exists a variety $Y$ with a regular action of $G$ and a birational $G$-equivariant morphism $\phi\colon X \dto Y$.
\end{thm}
We do not assume that $G$ is linear or connected, nor that $X$ is irreducible. This creates some complications in the arguments. The reader is advised to start with the case where $G$ is connected and $X$ irreducible in a first reading. 
\ps
We cannot expect that the birational map $\phi$ in the theorem is a morphism. Take the standard Cremona involution $\sigma$ of $\Ptwo$, given by $(x:y:z)\mapsto (\frac{1}{x},\frac{1}{y},\frac{1}{z})$, which collapses the coordinate lines to points. This cannot happen if $\sigma$ is a regular automorphism. However, removing these lines, we get $\kst \times \kst$ where $\sigma$ is a well-defined automorphism. 

More generally, consider the rational action of $G:=\PSLtwo\times\PSLtwo$ on $\Ptwo$ induced by the birational isomorphism $\Pone\times\Pone \dto \Ptwo$. Then neither an open set carries a regular $G$-action, nor $\Ptwo$ can be embedded into a variety $Y$ with a regular $G$-action. 

\ps
As we will see in the proof below, one first constructs a suitable open set $U \subset X$ where the rational action of $G$ has very specific properties, and then one shows that $U$ can be equivariantly embedded into a variety $Y$ with a regular $G$-action.

\ps
\subsection{Rational maps} 
We first have to define and explain the different notion used in the theorem above.
We refer to \cite{Bl2016Algebraic-structur} for additional material and more details. 

Recall that a {\it rational map} $\phi\colon X \dto Y$ between two varieties  $X,Y$  is an equivalence class of pairs $(U,\phi_U)$ where $U \subset X$ is an open dense subset and $\phi_U \colon U \to Y$  a morphism. Two such pairs $(U,\phi_{U})$ and $(V,\phi_{V})$ are equivalent if $\phi_{U}|_{U\cap V} = \phi_{V}|_{U\cap V}$. We say that $\phi$ is {\it defined in $x\in X$} if there is a $(U,\phi_{U})$ representing $\phi$ such that $x \in U$. The set of all these points forms an open dense subset $\Dom(\phi)\subseteq X$ called the {\it domain of definition of $\phi$}. We will shortly say that {\it $\phi$ is defined in $x$} if $x \in\Dom(\phi)$.

For all $(U,\phi_{U})$ representing $\phi\colon X \dto Y$ the closure $\overline{\phi_{U}(U)} \subset Y$ is the same closed subvariety of $Y$. We will call it the {\it closed image} of $\phi$ and denote it by $\overline{\phi(X)}$. 
The rational map $\phi$ is called {\it dominant\/} if $\overline{\phi(X)}=Y$. It follows that the composition $\psi\circ\phi$ of two rational maps $\phi\colon X \dto Y$ and $\psi\colon Y \dto Z$ is a well-defined rational map $\psi\circ\phi\colon X \dto Z$ in case $\phi$ is dominant.

A rational map $\phi\colon X \dto Y$ is called {\it birational\/} if it is dominant and admits an inverse $\psi\colon Y \dto X$, $\psi\circ\phi = \id_{X}$. It then follows that $\psi$ is also dominant and that $\phi\circ \psi = \id_{Y}$. Clearly, $\psi$ is well-defined by $\phi$, and we shortly write $\psi=\phi^{-1}$. It is easy to see that $\phi$ is birational if and only if there is a $(U, \phi_{U})$ representing $\phi$ such that $\phi_{U}\colon U \into Y$ is an open immersion with a dense image. The set of birational maps $\phi\colon X \dto X$ is a group under composition which will be denoted by $\Bir(X)$.

A rational map $\phi\colon X \dto Y$ is called {\it biregular in $x$\/} if there is an open neighborhood $U \subset \Dom(\phi)$ of $x$ such that $\phi|_{U}\colon U \into Y$ is an open immersion. It follows that the subset $X':=\{x \in X \mid \phi \text{ is biregular in }x\}$ is open in $X$, and the induced morphism $\phi\colon X' \into Y$ is an open immersion. This implies the following result.
\begin{lem}\label{Breg.lem}
Let $\phi\colon X \dto Y$  be a birational map. Then the set
$$
\Breg(\phi):=\{x \in X \mid \phi\text{ is biregular in }x\}
$$
is open and dense in $X$.
\end{lem}

\begin{rem}
If $X$ is irreducible, a rational dominant map $\phi\colon X \dto Y$ defines a $\bk$-linear inclusion $\phi^{*}\colon \bk(Y) \into \bk(X)$ of fields. Conversely, for every inclusion $\alpha \colon  \bk(Y) \into \bk(X)$ of fields there is a unique dominant rational map $\phi\colon X \dto Y$ such that $\phi^{*}=\alpha$.  In particular, we have an isomorphism $\Bir(X) \simto \Aut_{\bk}(\bk (X))$ of groups, given by $\phi\mapsto (\phi^{*})^{-1}$.
\end{rem}

\ps
\subsection{Rational group actions}
\begin{defn} Let $X,Z$ be varieties.
A map $\phi \colon Z \to \Bir(X)$ is called a {\it morphism} if there is an open dense set $U \subset Z \times X$ with the following properties:
\be
\item[(i)]
The induced map $(z,x)\mapsto \phi(z)(x)\colon U \to X$ is a morphism of varieties.
\item[(ii)]
For every $z \in Z$ the open set $U_{z}:=\{x\in X\mid (z,x)\in U\}$ is dense in $X$.
\item[(iii)]
For every $z \in Z$ the birational map $\phi(z)\colon X \dto X$ is defined in $U_{z}$.
\ee
\end{defn}
Equivalently, we have a rational map $\Phi\colon Z \times X \to X$ such that, for every $z \in Z$,
\be
\item[(i)] 
the open subset $\Dom(\Phi) \cap (\{z\}\times X) \subseteq \{z\}\times X$ is dense, and 
\item[(ii)]
the induced rational map $\Phi_{z}\colon X \dto X$, $x \mapsto\Phi(z,x)$, is birational.
\ee

This definition allows to define the \name{Zariski}-topology on $\Bir(X)$ in the following way. 

\begin{defn} 
A subset $S \subseteq \Bir(X)$ is {\it closed\/} if for every morphism $\rho\colon Z \to \Bir(X)$ the inverse image $\rho^{-1}(S) \subseteq Z$ is closed.
\end{defn}

Now we can define rational group actions on varieties. Let $G$ be an algebraic group and let $X$ be a variety.
\begin{defn}
A {\it rational action of $G$ on $X$} is a morphism $\rho\colon G \to \Bir(X)$ which is a homomorphism of groups. 
\end{defn}
As we have seen above this means that we have a rational map (denoted by the same letter) $\rho\colon G\times X \dto X$ such that the following holds:
\be
\item
$\Dom(\rho)\cap (\{g\}\times X)$ is dense in $\{g\}\times X$ for all $g\in G$, 
\item
the induced rational map $\rho_{g}\colon X \dto X$, $x\mapsto \rho(g,x)$, is birational, 
\item
the map $g\mapsto \rho_{g}$ is a homomorphism of groups.
\ee

If $\rho$ is defined in $(g,x)$ and $\rho(g,x)= y$ we will say that {\it $g\cdot x$ is defined and $g\cdot x = y$.}

We will also use the birational map
$$
\tilde\rho\colon G \times X \dto G \times X, \quad (g,x)\mapsto (g,\rho(g,x)),
$$
see section~\ref{G-regular-points.subsec} below.

\begin{rem}
Note that if $\rho\colon G \times X \dto X$ is defined in $(g,x)$, then $\rho_{g}\colon X \dto X$ is defined in $x$, but  the reverse implication does not hold. An example is the following. Consider the regular action of the additive group $\Ga$ on the plane $\Atwo=\bk^{2}$ by translation along the $x$-axis: $s\cdot x := x +(s,0)$ for $s \in \Ga$ and $x \in \Atwo$. Let $\beta\colon X\to \Atwo$ be the blow-up of $\Atwo$ in the origin. Then we get a rational $\Ga$-action on $X$, $\rho\colon \Ga\times X \dto X$. It is not difficult to see that  $\rho$ is defined in $(e,x)$ if and only if $\beta(x)\neq 0$, i.e. $x$ does not belong to the exceptional fiber, but clearly, $\rho_{e}=\id$ is defined everywhere.
\end{rem}

If $\phi\colon Z \to \Bir(X)$ is a morphism such that $\phi(Z) \subset \Aut(X)$, the group of regular automorphisms, one might conjecture that 
the induced map $Z \times X \to X$ is a morphism. I don't know how to prove this, but maybe the following holds.

\begin{conj*}
Let $\rho \colon G \to \Bir(X)$ be a rational action. If $\rho(G) \subset \Aut(X)$, then $\rho$ is a regular action.
\end{conj*}
We can prove this under additional assumptions.

\begin{thm}\label{rational-regular.thm}
Let $\rho\colon G\to\Bir(X)$ be a rational action where $X$ is affine. Assume that there is a dense subgroup $\Gamma \subset G$ such that $\rho(\Gamma) \subset\Aut(X)$. Then the $G$-action on $X$ is regular.
\end{thm}
The proof will be given in the last section~\ref{proof.subsec}.

\begin{defn}
Given rational $G$-actions $\rho$ on $X$ and $\mu$ on $Y$, a dominant rational map $\phi\colon X \to Y$ is called {\it $G$-equivariant\/} if the following holds:
\begin{itemize}
\item[(Equi)]
For every $(g,x)\in G\times X$ such that (1) $\rho$ is defined in $(g,x)$, (2) $\phi$ is defined in $x$ and in $\rho(g,x)$, and  (3) $\mu$ is defined in $(g,\phi(x))$, we have $\phi(\rho(g,x)) = \mu(g,\phi(x))$.
\end{itemize}
\end{defn}
Note that the set of $(g,x)\in G\times X$ satisfying the assumptions of (Equi) is open and dense in $G \times X$ and has the property that it meets all $\{g\}\times X$ in a dense open set.

\begin{rem}\label{open-subsets.rem}
If $G$ acts rationally on $X$ and if $X' \subset X$ is an nonempty open subset, then $G$ acts rationally on $X'$, and the inclusion $X' \into X$ is $G$-equivariant. Moreover, if $G$ acts rationally on $X$ and if $\phi\colon X \dto Y$ is a birational map, then there is uniquely define rational action of $G$ on $Y$ such that $\phi$ is $G$-equivariant.

Note that for a rational $G$-action $\rho$ on $X$ and an open dense set $X '\subset X$ with induced rational $G$-action $\rho'$ we have 
\begin{gather*}
\Dom(\rho')=\{(g,x)\in\Dom(\rho) \mid x \in X' \text{ and }g\cdot x \in X'\},\\
\Breg(\trho')=\{(g,x)\in\Breg(\trho) \mid x \in X' \text{ and }g\cdot x \in X'\}.
\end{gather*}
\end{rem}

\ps
\subsection{The case of a finite group \texorpdfstring{$G$}{G}}
Assume that $G$ is finite and acts rationally on an irreducible variety $X$. Then every $g \in G$ defines a birational map $g\colon X \dto X$ and thus an automorphism $g^{*}$ of the field $\bk(X)$ of rational functions on $X$. In this way we obtain a homomorphism $G \to \Aut_{\bk}(\bk(X))$ given by $g \mapsto (g^{*})^{-1}$. 

By Remark~\ref{open-subsets.rem} above we may assume that $X$ is affine. Hence $\bk(X)$ is the field of fractions of the coordinate ring $\OOO(X)$. Since $G$ is finite we can find a finite-dimensional $\bk$-linear subspace $V \subset \bk(X)$ which is $G$-stable and contains a system of generators of $\OOO(X)$.

Denote by $R \subset \bk(X)$ the subalgebra generated by $V$. By construction, 
\be
\item $R$ is finitely generated and $G$-stable, and 
\item $R$ contains $\OOO(X)$. 
\ee
In particular, the field of fractions of $R$ is $\bk(X)$.
If we denote by $Y$ the affine variety with coordinate ring $R$, we obtain a regular action of $G$ on $Y$ and a birational morphism $\psi\colon Y \to X$ induced by the inclusion $\OOO(X) \subset R$. Now the Regularization Theorem follows in this case with $\phi:=\psi^{-1}\colon X \dto Y$.

\ps
There is a different way to construct a ``model'' with a regular $G$-action, without assuming that $X$ is irreducible. In fact, there is always an open dense set $\Xr \subset X$ where the action is regular. It is defined in the following way (cf. Definition~\ref{G-regular.def} below). For $g \in G$ denote by $X_{g}\subset X$ the open dense set where the rational map $\rho_{g}\colon x \mapsto g\cdot x$ is biregular. Then $\Xr:=\bigcap_{g\in G} X_{g}$ is open and dense in $X$ and the rational $G$-action on $\Xr$ is regular. In fact, $\rho_{g}$ is biregular on $\Xr$, hence also biregular on $h \cdot\Xr$ for all $h \in G$ which implies that $h\cdot \Xr \subset \Xr$.

\ps
\subsection{A basic example}
We now give an example which should help to understand the constructions and the proofs below. Let $X$ be a variety with a regular action of an algebraic group $G$. Choose an open dense subset $U \subset X$ and consider the rational $G$-action on $U$. Then $\tilde X := \bigcup_{g\in G}g U\subseteq X$ is open and dense in $X$ and carries a regular action of $G$.

The rational $G$-action $\rho$ on $U$ is rather special. First of all we see that $\rho$ is defined in $(g,u)$ if and only if $g\cdot u \in U$. This implies that $\rho$ is defined in $(g,u)$ if and only of $\rho_{g}$ is defined in $u$. Next we see that if $\rho$ is defined in $(g,u)$, then $\tilde\rho\colon G \times U \dto G\times U$, $(g,x)\mapsto (g,\rho(g,x))$,  is biregular in $(g,u)$. And finally, for any $x$ the (open) set of elements $g \in G$ such that $\trho$ is biregular in $(g,x)$ is dense in $G$.

\ps
A first and major step  in the proof  is to show (see section~\ref{G-regular-points.subsec}) that for every rational $G$-action on a variety $X$ there is an open dense subset $\Xr\subset X$ with the property that for every $x\in \Xr$ the rational map $\tilde\rho\colon G\times \Xr \dto G \times \Xr$ is biregular in $(g,x)$ for all $g$ in a dense (open) set of $G$. Then, in a second step in section~\ref{model.subsec}, we construct from $\Xr$ a variety $Y$ with a regular $G$-action together with an open $G$-equivariant embedding $\Xr \into Y$.

\ps
\subsection{\texorpdfstring{$G$}{G}-regular points and their properties}\label{G-regular-points.subsec}
Let $X$ be a variety with a rational action $\rho\colon G \times X \dto X$ of an algebraic group $G$.
Define 
$$
\tilde\rho\colon G \times X \dto G \times X, \quad  (g,x)\mapsto (g,\rho(g,x)).
$$
It is clear that $\Dom(\tilde\rho) = \Dom(\rho)$ and that $\tilde\rho$ is birational with inverse $\tilde\rho^{-1}(g,y) =(g,\rho(g^{-1},y))$, i.e. $\tilde\rho^{-1}= \tau\circ\tilde\rho\circ\tau$ where $\tau\colon G \times X \simto G\times X$ is the isomorphism $(g,x) \mapsto (g^{-1},x)$. 

The following definition is crucial. 
\begin{defn}\label{G-regular.def}
A point $x \in X$ is called {\it $G$-regular} for the rational $G$-action $\rho$ on $X$ if $\Breg(\trho)\cap (G\times\{x\})$ is dense in $G\times\{x\}$, i.e. $\trho$ is biregular in $(g,x)$ for all $g$ in a dense (open) set of  $G$. 
We denote by $\Xr \subset X$ the set of $G$-regular points.
\end{defn}

Let $\lambda_{g}\colon G \simto G$ be the left multiplication with $g \in G$. For every $h \in G$ we have the following commutative diagram
\begin{center}
\begin{tikzcd}
G \times X  \arrow[r, dashed, "\tilde\rho"]\arrow[d,"\lambda_h\times\id"',"\simeq"]
  & G \times X \arrow[d,dashed,"\lambda_h\times\rho_h"]
\\
G\times X\arrow[r,dashed,"\tilde\rho"]
  & G \times X
\end{tikzcd}
\end{center}
This implies the following result.

\begin{lem}\label{biregular.lem}
With the notation above we have the following:
\be
\item\label{a}
If $\rho$ is defined in $(g,x)$ and $\rho_h$ defined in $g\cdot x$, then $\rho$ is defined in $(h g,x)$.
\item\label{b}
If $\tilde\rho$ is biregular in $(g,x)$ and $\rho_h$ biregular in $g\cdot x$, then $\tilde\rho$ is biregular in $(h g,x)$.
\ee
\end{lem}
The main proposition is the following.
\begin{prop}\label{main.prop}
\be
\item\label{aa}
$\Xr$ is open and dense in $X$.
\item\label{bb}
If $x \in \Xr$ and if  $\tilde\rho$ is biregular in $(g,x)$, then $g\cdot x \in \Xr$.
\ee
\end{prop}
\begin{proof}
(a) Let $G = G_0\cup G_1\cup\cdots\cup G_n$ be the decomposition into connected components. Then $D_i:= \Breg(\rho)\cap (G_{i}\times X)$ is open and dense for all $i$ (Lemma~\ref{Breg.lem}), and the same holds for the image $\bar D_{i} \subset X$ under the projection onto $X$. Since $\Xr=\bigcap_{i}\bar D_{i}$, the claim follows.
\ps
(b) If $\tilde\rho$ is biregular in $(g,x)$, then $\trho^{-1} = \tau\circ\trho\circ\tau$ is biregular in $(g,g\cdot x)$, hence $\trho$ is biregular in $\tau(g,g\cdot x) = (g^{-1},g\cdot x)$.
If $x$ is $G$-regular, then $\rho_{h}$ is biregular in $x$ for all $h$ from a dense open set $G'\subset G$. Now Lemma~\ref{biregular.lem}(\ref{b}) implies that $\tilde\rho$ is biregular in $(h g^{-1},g\cdot x)$ for all $h\in G'$, hence $g\cdot x\in\Xr$.
\end{proof}
Note that for an open dense set $U \subset X$ a point $x \in U$ might be $G$-regular for the rational $G$-action on $X$, but not for the rational $G$-action on $U$. However, Proposition~\ref{main.prop}(b) implies the following result.
\begin{cor}\label{reduction-to-Xreg.cor}
For the rational $G$-action on $\Xr$ every point is $G$-regular.
\end{cor}
This allows to reduce to the case of a rational $G$-action such every point is $G$-regular.
\begin{lem}\label{def-bireg.lem}
Assume that  $X=\Xr$.
If $\rho_{g}$ is defined in $x$, then $\rho_{g}$ is biregular in $x$. 
\end{lem}
\begin{proof}
\ps
Assume that $\rho_{g}$ is defined in $x \in X$. There is an open dense subset $G' \subset G$ such $\rho_{h}$ is biregular in $g\cdot x$ and $\rho_{hg}$ is biregular in $x$ for all $h \in G'$. Since $\rho_{hg} = \rho_{h} \circ \rho_{g}$ we see that $\rho_{g}$ is biregular in $x$.
\end{proof}

For a rational map $\phi\colon X \dto Y$ the {\it graph} $\Gamma(\phi)$ is defined in the usual way:
$$
\Gamma(\phi):=\{(x,y) \in  X \times Y \mid \phi \text{ is defined in $x$ and }\phi(x)=y\}.
$$
In particular, $\pr_{X}(\Gamma(\phi)) =\Dom(\phi)$ and $\pr_{Y}(\Gamma(\phi)) = \phi(\Dom(\phi))$.
\ps
The next lemma will play a central r\^ole in the construction of the regularization.

\begin{lem}\label{closed-graph.lem}
Consider a rational $G$-action $\rho$ on a variety $X$ and assume that every point of $X$ is $G$-regular. Then, for every $g \in G$, the graph $\Gamma(\rho_{g})$ is closed in $X \times X$.
\end{lem}

\begin{proof} Let $\Gamma:=\overline{\Gamma(\rho_{g})}$  be the closure of the graph of $\rho_{g}$ in $X \times X$. We have to show that for every $(x_{0},y_{0})\in \Gamma$ the rational map $\rho_{g}$ is defined in $x_{0}$, or, equivalently, that the morphism $\pi_{1}:=\pr_{1}|_{\Gamma}\colon\Gamma \to X$ induced by the first projection is biregular in $(x_{0},y_{0})$. 

Choose $h \in G$ such that $\rho_{h g}$ is biregular in $x_{0}$ and $\rho_{h}$ is biregular in $y_{0}$, and consider the induced birational map $\Phi:=(\rho_{h g}\times\rho_{h})\colon X \times X \dto X \times X$. 
If $\Phi$ is defined in $(x,y) \in \Gamma(\rho_{g})$, $y:=g\cdot x$, then $\Phi(x,y) = ((h g)\cdot x, (h g)\cdot x) \in \Delta(X)$ where $\Delta(X):=\{(x,x) \in X \times X \mid x \in X\}$ is the diagonal. It follows that $\overline{\Phi(\Gamma)}\subseteq \Delta(X)$. 
\begin{center}
\begin{tikzcd}
X \times X  \arrow[r, dashed, "\rho_{h g}\times\rho_{h}"]
  & X \times X
\\
\Gamma\arrow[r,dashed,"\phi"] \arrow[u,hook, "\subset"'] \arrow[d,"\pi_1"']
  & \Delta(X)\arrow[u,hook,"\subset"'] \arrow[d, "\pr_1"', "\simeq"]
\\
 X  \arrow[r,dashed]{}{\rho_{h g}}
  & X
\end{tikzcd}
\end{center}
Since $\Phi$ is biregular in $(x_{0},y_{0})$, we see that  $\phi:=\Phi|_{\Gamma}\colon \Gamma \dto \Delta(X)$ is also biregular in $(x_{0},y_{0})$. 
By construction, we have $\rho_{h g}\circ\pi_{1}=\pr_{1}\circ\phi$. Since $\rho_{h g}$ is biregular in $\pi_{1}(x_{0},y_{0})$ and $\phi$ is biregular in $(x_{0},y_{0})$ (and $\pr_{1}|_{\Delta(X)}$ is an isomorphism)  it follows that $\pi_{1}$ is biregular in $(x_{0},y_{0})$, hence the claim.
\end{proof}

The last lemma is easy.

\begin{lem}\label{last.lem}
Consider a rational action $\rho$ of $G$ on a variety $X$. Assume that there is a dense open set $U \subset X$ such that $\tilde\rho$ defines an open immersion $\tilde\rho\colon G \times U \into G \times X$. Then the open dense subset $Y:=\bigcup_{g}g\cdot U \subset X$ carries a regular $G$-action.
\end{lem}
\begin{proof}
It is clear that every $\rho_{g}$ induces an isomorphism $U \simto g\cdot U$. This implies that $Y$ is stable under all $\rho_{g}$. It remains to see that the induced map $G \times Y \to Y$ is a morphism. By assumption, this is clear on $G \times U$, hence also on $G \times g\cdot U$ for all $g \in G$, and we are done.
\end{proof}

\ps
\subsection{The construction of a regular model}\label{model.subsec}
In view of Corollary~\ref{reduction-to-Xreg.cor} our Theorem~\ref{thm1} will follow from the next result.
\begin{thm}\label{thm3}
Let $X$ be a variety with a rational action of $G$. Assume that every point of $X$ is $G$-regular. Then there is a variety $Y$ with a regular $G$-action and a $G$-equivariant open immersion $X \into Y$.
\end{thm}
From now on $X$ is a variety with a rational $G$-action $\rho$ such that $\Xr = X$.
Let $S:=\{g_{0}:=e, g_{1},g_{2},\ldots,g_{m}\} \subset G$ be a finite subset. These $g_{i}$'s will be carefully chosen in 
the proof of Theorem~2 below. Let $\Xnull,\Xone,\ldots,\Xm$ be copies of the variety $X$.  On the disjoint union $X(S) := \Xnull\cup\Xone\cup\cdots\cup\Xm$ we define the following relations between elements $x_{i},x_{i}'\in \Xi$:
\begin{gather}
\text{ For any $i$: \ } x_{i}\sim x_{i}' \ \iff \ x_{i}=x_{i}';\\
\text{For $i\neq j$: \ }x_{i} \sim x_{j} \ \iff \ \text{$\rho_{g_{j}^{-1}g_{i}}$  is defined in $x_{i}$  and sends $x_{i}$ to $x_{j}$.}
\end{gather}
It is not difficult to see that this defines an equivalence relation. (For the symmetry one has to use Lemma~\ref{def-bireg.lem}.)
Denote by $\bX(S) := X(S)/\sim$ the set of equivalence classes endowed with the induced topology.

\begin{lem}\label{construction.lem}
The maps $\iota_{i}\colon \Xi \to \bX(S)$ are open immersions and endow $\bX(S)$ with the structure of a variety.
\end{lem}
\begin{proof}
By definition of the equivalence relation and the quotient topology the natural maps $\iota_{i}\colon \Xi \to \bX(S)$ are injective and continuous. Denote the image by $\bXi$. We have to show that $\bXi$ is open in $\bX(S)$, or, equivalently, that the inverse image of $\bXi$ in $X(S)$ is open. This is clear, because the inverse image in $\Xi$ of the intersection $\bXi\cap\bXj$ is the open set of points where 
$\rho_{g_{j}^{-1}g_{i}}$ is defined. 

It follows  that $\bX(S)$ carries a unique structure of a prevariety such that the maps $\iota_{i}\colon \Xi \into \bX(S)$ are open immersions. 
It remains to see that the diagonal $\Delta(\bX(S)) \subset \bX(S) \times \bX(S)$ is closed. For this it suffices to show that $\Delta_{i j}:=\Delta(\bX(S))\cap (\bXi\times\bXj)$ is closed in $\bXi\times\bXj$ for all $i,j$. This follows from Lemma~\ref{closed-graph.lem}, because $\Delta_{i j}$ is the image of $\Gamma(\rho_{g_{j}^{-1}g_{i}}) \subset \Xi\times\Xj$. In fact, for $x_{i}\in \Xi$ and $x_{j}\in \Xj$ we have $(\bar x_{i},\bar x_{j})\in\Delta_{i j}$ if and only if $x_{i}\sim x_{j}$. This means that $\rho_{g_{j}^{-1}g_{i}}$ is defined in $x_{i}$ and $\rho_{g_{j}^{-1}g_{i}}(x_{i})=x_{j}$, i.e. $(x_{i},x_{j}) \in \Gamma(\rho_{g_{j}^{-1}g_{i}})$. 
\end{proof}

Fixing the open immersion $\iota_{0}\colon X = \Xnull \into \bX(S)$ we obtain a rational $G$-action  $\brho=\brho_{S}$ on $\bX(S)$ such that  $\iota_{0}$ is $G$-equivariant (Remark~\ref{open-subsets.rem}). 
If we consider each $\Xi$ as the variety $X$ with the rational $G$-action $\rhoi(g,x) := \rho(g_{i}g g_{i}^{-1},x)$, then, by construction of $\bX(S)$, the open immersions $\iota_{i}\colon \Xi \into \bX(S)$ are all $G$-equivariant. 

\begin{lem}\label{gi-is-defined-on-Xnull.lem}
For all $i$, the rational map $\brho_{g_{i}}$ is defined on $\bXnull$ and defines an isomorphism $\brho_{g_{i}}\colon\bXnull \simto \bXi$.
\end{lem}
\begin{proof} Consider the open immersion $\tau_{i}:=\iota_{i}\circ\iota_{0}^{-1}\colon\bXnull \into \bX(S)$ with image $\bXi$. We claim that $\tau_{i}(\bar x) = g_{i}\cdot \bar x$. It suffices to show that this holds on an open dense set of $\bXnull$. Let $U \subset X$ be the open dense set where $g_{i}\cdot x$ is defined. For $x \in U$ and $y:=g_{i}\cdot x \in X$ we get, by definition, $\iota_{0}(y) = \iota_{i}(x)$. On the other hand, $\iota_{0}(y) = \iota_{0}(g_{i}\cdot x) = g_{i}\cdot \iota_{0}(x)$. Hence, $g_{i}\cdot\iota_{0}(x) = \iota_{i}(x)$, and so $\tau_{i}(\bar x) = g_{i}\cdot \bar x$ for all $\bar x \in \iota_{0}(U)$.
\end{proof}
\newcommand{\Dnull}{D^{(0)}}
\begin{proof}[Proof of Theorem~\ref{thm3}]
(a)
Since $\Xr = X$, we see that for any $x \in X$ there is a $g\in G$ such that $(g,x)\in D$, hence $\bigcup_{g}g D = G\times X$ where $G$ acts on $G \times X$ by left-multiplication on $G$. As a consequence, we have $\bigcup_{i}g_{i}D = G \times X$ for a suitable finite subset $S = \{g_{0}=e,g_{1},\ldots,g_{m}\} \subset G$. This set $S$ will be used to construct $\bX(S)$.
\ps
(b)
Let $\Dnull \subset G \times \bXnull$ be the image of $D$, and consider the rational map $\tilde\rho_{S}\colon G\times \bXnull \dto G \times \bX(S)$, $(g,\bar x)\mapsto (g,\brho(g,\bar x))$. We claim that $\tilde\rho_{S}$ is biregular. In fact, for any $i$, the map $(g,x) \mapsto (g, g\cdot x)$ is the composition of $(g,x)\mapsto (g,(g_{i}^{-1}g)x)$ and $(g,y)\mapsto (g,g_{i}y)$ where the first one is biregular on $g_{i}\Dnull$ with image in $G \times \bXnull$, and the second is biregular on $G \times \bXnull$, by Lemma~\ref{gi-is-defined-on-Xnull.lem}. Now the claim follows, because $G\times \bXnull = \bigcup_{i} g_{i}\Dnull$, by (a).
\ps
(c)
It follows from (b) that the rational action $\brho$ of $G$ on $\bX(S)$ has the property, that $\tilde\rho_{S}$ defines an open immersion $G \times \bXnull \into G\times \bX(S)$. Now Theorem~\ref{thm3} follows from Lemma~\ref{last.lem}, setting $Y:= \bX(S)$.
\end{proof}

\ps
\subsection{Normal and  smooth models}
If $X$ is an irreducible $G$-variety, i.e. a variety with a regular action of $G$, then it is well-known that the normalization $\tilde X$ has a unique structure of a $G$-variety such that the normalization map $\eta\colon \tilde X \to X$ is $G$-equivariant. If $X$ is reducible, $X = \bigcup_{i} X_{i}$, we denote by $\tilde X$ the disjoint union of the normalizations of the irreducible components $X_{i}$, $\tilde X = \dot\bigcup_{i}\tilde X_{i}$, and by $\eta\colon \tilde X \to X$ the obvious morphism which will be called the {\it normalization of $X$}. The proof of the following assertion is not difficult.
\begin{prop}
Let $X$ be a $G$-variety and $\eta\colon \tilde X \to X$ its normalization. Then there is a unique regular $G$-action on $\tilde X$ such that $\eta$ is $G$-equivariant.
\end{prop}
It is clear that for any $G$-variety $X$ the open set $X_{\text{\it smooth}}$ of smooth points is stable under $G$. Thus smooth models for a rational $G$-action always exist.

The next result, the {\it equivariant resolution of singularities}, can be found in \name{Koll\'ar}'s book \cite{Ko2007Lectures-on-resolu}. He shows in Theorem~3.36 that in characteristic zero there is a functorial resolution of singularities $\BR(X) \colon X' \to X$ which commutes with surjective smooth morphisms. This implies (see his Proposition~3.9.1) that every action of an algebraic group on $X$ lifts uniquely to an action on $X'$.
\begin{prop}
Assume $\Char \bk = 0$, and 
let $X$ be a $G$-variety. Then there is a smooth $G$-variety $Y$ and a proper birational $G$-equivariant morphism $\phi\colon Y \to X$.
\end{prop}

\subsection{Projective models}
The next results show that there are always smooth projective models for connected algebraic groups $G$. More precisely, we have the following propositions.

\begin{prop}
Let $G$ be a connected algebraic group acting on a normal variety $X$. Then there exists an open cover of $X$ by quasi-projective $G$-stable varieties.
\end{prop}

\begin{prop}
Let $G$ be a connected algebraic group acting on a normal quasi-projective variety $X$. Then there exists a $G$-equivariant embedding into a projective $G$-variety.
\end{prop}
\begin{proof}[Outline of Proofs]
Both propositions are due to \name{Sumihiro} in case of a connected linear algebraic group $G$ \cite{Su1974Equivariant-comple,Su1975Equivariant-comple}. 
They were generalized to a connected algebraic group $G$ by \name{Brion} in \cite[Theorem~1.1 and Theorem~1.2]{Br2010Some-basic-results}.
\end{proof}
In this context let us mention the following 
{\it equivariant \name{Chow}-Lemma}. For a connected linear algebraic group $G$ it was proved by \name{Sumihiro} \cite{Su1974Equivariant-comple} and later generalized to the non-connected case by \name{Reichstein-Youssin} \cite{ReYo2002Equivariant-resolu}. It implies that projective models always exist for linear algebraic groups $G$.

\begin{prop}[\protect{\cite[Theorem~2]{Su1974Equivariant-comple}, \cite[Proposition~2]{ReYo2002Equivariant-resolu}}] Let $G$ be a linear algebraic group. 
For every $G$-variety $X$ there exists a quasi-projective $G$-variety $Y$ and a proper birational $G$-equivariant morphism $Y \to X$ which is an isomorphism on a $G$-stable open dense subset $U \subset Y$.
\end{prop}

\ps
\subsection{Proof of Theorem~\ref{rational-regular.thm}}\label{proof.subsec}
We start with a rational action $\rho\colon G \to \Bir(X)$ of an algebraic group $G$ on a variety $X$, and we assume that there is a dense subgroup $\Gamma \subseteq G$ such that $\rho(\Gamma) \subset \Aut(X)$. 
\ps
(a) 
We first claim that the rational $G$-action on the open dense set $\Xr\subset X$ is regular. For every $x \in \Xr$ there is a $g \in \Gamma$ such that $\tilde\rho$ is biregular in $(g,x)$. Since, by assumption, the $\rho_{h}$ are biregular on $X$ for all $h \in \Gamma$ it follows from Lemma~\ref{biregular.lem}(\ref{b}) that $\tilde\rho$ is biregular in $(g',x)$ for any $g' \in \Gamma$. Moreover,  by Proposition~\ref{main.prop}(\ref{bb}), we have $g'\cdot x \in \Xr$, hence $\Xr$ is stable under $\Gamma$. 

By Theorem~\ref{thm3} we have a $G$-equivariant open immersion $\Xr \into Y$ where $Y$ is a variety with a regular $G$-action. Since the complement $C:=Y \setminus \Xr$ is closed and $\Gamma$-stable we see that $C$ is stable under $\bar\Gamma = G$, hence the claim.
\ps
(b)
From (a) we see that the rational map $\rho\colon G \times X \dto X$ has the following properties:
\be
\item[(i)] There is a dense open set $\Xr \subset X$ such that $\rho$ is regular on $G \times \Xr$.
\item[(ii)] For every $g \in \Gamma$ the rational map $\rho_{g}\colon X \to X$, $x\mapsto\rho(g,x)$, is a regular isomorphism.
\ee
Now the following lemma implies that $\rho$ is a regular action in case $X$ is affine, proving Theorem~\ref{rational-regular.thm}.
\qed
\begin{lem}\label{lem1}
Let $X,Y,Z$ be varieties and let $\phi\colon X \times Y \dto Z$ be a rational map where $Z$ is affine. Assume the following:
\be
\item There is an open dense set $U \subset Y$ such that $\phi$ is defined on $X \times U$;
\item There is a dense set $X' \subseteq X$ such that the induced maps $\phi_{x}\colon \{x\}\times Y \to Z$ are morphisms for all $x \in X'$
\ee
Then $\phi$ is a regular morphism.
\end{lem}

\begin{proof} 
We can assume that $Z=\Aone$, so that $\phi=F$ is a rational function on $X \times Y$. We
can also assume that $X,Y$ are affine and that $U = Y_{f}$ with a non-zero divisor $f \in \OOO(Y)$. This implies that $f^{k}F \in \OOO(X \times Y) = \OOO(X)\otimes\OOO(Y)$ for some $k \geq 0$. Write $f^{k}F = \sum_{i=1}^{n} h_{i}\otimes f_{i}$ with 
$\bk$-linearly independent $h_{1},\ldots,h_{n}\in\OOO(X)$. Setting $F_{x}(y):=f(x,y)$ for $x \in X$, the assumption implies that $F_{x} = \sum_{i = 1}^{n}h_{i}(x) \frac{f_{i}}{f^{k}}$ is a regular function on $Y$ for all $x \in X'$. 

We claim that there exist $x_{1},\ldots,x_{n}\in X'$ such that the $n\times n$-matrix $(h_{i}(x_{j}))_{i,j=1}^{n}$ is invertible. This implies that the rational functions $\frac{f_{i}}{f^{k}}$ are $\bk$-linear combinations of the $F_{x_{i}}=f(x_{i},y)\in\OOO(Y)$. Hence they are regular, and thus $F$ is regular. The lemma follows.

It remains to prove the claim.
Assume that we have found $x_{1},\ldots,x_{m} \in X'$ ($m<n$) such that the $m\times m$-matrix $(h_{i}(x_{j}))_{i,j=1}^{m}$ is invertible. Then there are uniquely defined $\lambda_{1},\ldots,\lambda_{m}\in \bk$ such that $h_{m+1}(x_{i}) = \sum_{j=1}^{m}\lambda_{j}h_{j}(x_{i})$ for $i=1,\ldots,m$. Since $h_{1},\ldots,h_{m},h_{m+1}$ are linearly independent, it follows that  $h_{m+1}\neq \sum_{j=1}^{m}\lambda_{j}h_{j}$. This implies that there exists $x_{m+1}\in X'$ such that $h_{m+1}(x_{m+1})\neq \sum_{j=1}^{m}\lambda_{j}h_{j}(x_{m+1})$, and so the matrix  $(h_{i}(x_{j}))_{i,j=1}^{m+1}$ is invertible. Now the claim follows by induction.
\end{proof}

\bigskip\bigskip
\end{document}